\documentclass[11pt]{amsart}

\usepackage{amssymb}
\usepackage{lmodern}
\usepackage{mathtools}
\usepackage{hyperref}
\usepackage{mathrsfs}
\mathtoolsset{showonlyrefs}
\usepackage{indentfirst} 

\hypersetup{colorlinks=true,urlcolor=black,  pdftitle={Convexity of Radial Mean Bodies via an Extension of Ball's Bodies},
    pdfauthor={Dylan Langharst}}

\newtheorem{conj}{Conjecture}
\newtheorem{thm}[conj]{Theorem}
\newtheorem{lem}[conj]{Lemma}
\newtheorem{prop}[conj]{Proposition}
\newtheorem{cor}[conj]{Corollary}
\newtheorem{ques}{Question}
\newtheorem{defn}[conj]{Definition}

\newcommand{\dlat}{\mathrm{d}}

\newcommand{\vol}{\mathrm{Vol}}
\newcommand{\supp}{\mathrm{supp}}
\newcommand{\N}{\mathbb{N}}

\def\s{\mathbb{S}}
\def\E{\mathbb{E}}
\def\R{{\mathbb R}}

\def\phi{\varphi}
\def\LC{\operatorname{LC}_n}

\begin{document}

\title[Convexity of Radial Mean Bodies]{Convexity of Radial Mean Bodies via \\
an extension of Ball's Bodies}

\author{Dylan Langharst}
\address{Department of Mathematical Sciences \\
Carnegie Mellon University \\
Wean Hall, Pittsburgh, PA 15213, USA}
\email{dlanghar@andrew.cmu.edu}

\subjclass[2020]{Primary 52A20; Secondary 52A30, 26B25}
\keywords{Log-concave functions, radial mean bodies, Ball's bodies, Pr\'ekopa's theorem}

\begin{abstract}
In this work, we extend a classical theorem of Keith Ball on integrals of log-concave functions along rays against the weight $r^{p-1}$ to the previously inaccessible regime $p\in (-1,0)$: if $g:\mathbb R^n\to\mathbb R_+$ is an integrable, upper semi-continuous, log-concave function which attains its maximum at the origin and has the origin in the interior of its support, then
\[
x\mapsto \left(\frac{p}{g(o)}\int_{0}^{\infty}r^{p-1}(g(rx)-g(o))\mathrm{d}\,r\right)^{-\frac{1}{p}}
\]
is a gauge on $\mathbb{R}^n$. Our approach also provides a new proof of the original regime $p> 0$. The argument is based on a reduction to a two-dimensional inequality derived from Pr\'ekopa's theorem, which may be of independent interest.

As a consequence of this extension, we resolve a nearly 30-year-old question of Richard Gardner and Gaoyong Zhang in the affirmative. In 1998, R. Gardner and G. Zhang introduced the radial $p$th mean bodies $R_p K$ of a convex body $K\subset \mathbb{R}^n$ for $p>-1$. Furthermore, they established that $R_p K$ is convex for $p\geq 0$, but the convexity of $R_p K$ for $p\in (-1,0)$ remained open. We prove that $R_p K$ is convex for all $p>-1$. \end{abstract}

\maketitle

\section{Introduction}
We denote by $\R^n$ the $n$-dimensional Euclidean space with its usual structure and $\R_+=[0,\infty)$. We recall that a function $f:\R^n\rightarrow \R_+$ is said to be $\log$-concave if for every $\lambda \in (0,1)$ and $x,y$ such that $f(x)f(y)>0$, one has 
\begin{equation}f((1-\lambda)x + \lambda y) \geq f(x)^{1-\lambda}f(y)^\lambda.
\label{eq:log}
\end{equation}
The study of convex bodies (compact, convex sets with non-empty interior) associated to log-concave functions has played a central role in modern convex geometry. A fundamental example is the construction introduced by Keith Ball. We define the following class of log-concave functions:
\begin{align*}
\LC:=&\left\{f:\R^n\rightarrow \R_+: f\text{ is log-concave,}\right.
\\
&\qquad\qquad\left.\text{upper semi-continuous, and} \; \; 0\!<\!\int_{\R^n} f(x) \dlat x \!< \infty\right\},
\end{align*}
and the subclass
\[
\LC^0:=\left\{g\in\LC:\max g=g(o)\text{ and }o\in\operatorname{int}\left(\supp(g)\right)\right\},
\]
where
\[
\supp(g)=\overline{\{x\in\R^n:g(x)>0\}}.
\]
For every $g\in\LC$, its support is a potentially unbounded convex set with non-empty interior. Ball's construction assigns to $g\in\LC^0$ a family $K_p(g)$ of convex bodies depending on a parameter $p>0$ via integration along rays with weight $r^{p-1}$. 

This construction has proven to be a powerful tool, with connections to sections of convex bodies \cite{FMY17,KZ15} and log-concave functions \cite{NT26,GT26}, affine isoperimetric inequalities of Blaschke-Santalo \cite{HL17,FM07,CEFPP15}, Zhang's projection \cite{GZ98,ABG20,LP25,LRZ22,LPRY25}, Rogers-Shephard \cite{AHNRZ21,AG19,AlGMJV,HLPRY25} and Brunn-Minkowski type \cite{AG24,KM05}. Ball's bodies have also appeared in geometric functional analysis, such as in the study of Bourgain's slicing problem \cite{BK05,BK06,GPV12} and concentration and thin shell phenomena \cite{GP06,FGP14,PP13-2,GM11}.

In this work, we show that K. Ball's construction extends to the full range $p>-1$. More precisely, we extend the definition of $K_p(g)$ to $p\in (-1,0]$ for functions $g\in\LC^0$ (see Definition~\ref{def:keith_ball_bodies} below), and prove that these sets are convex. This provides a unified framework encompassing several previously studied objects.
\begin{thm}
\label{t:ball_expanded}
    Let $g\in\LC^0$. Then, for every $-1<p<\infty$, $K_p(g)$ is a convex body containing the origin in its interior; equivalently, $\|\cdot\|_{K_p(g)}$ is a finite gauge that is positive on $\R^n\setminus\{o\}$.
\end{thm}

Our approach reduces the problem to a two-dimensional inequality derived from Pr\'ekopa's theorem, which states that marginals of log-concave functions are log-concave.
\begin{prop}[Pr\'ekopa's theorem, \cite{PreL1}]
\label{p:pre}
    Let $f:\R^{n}\times \R^{m}\to \R_+$ be a log-concave, non-identically zero function. Then, if the function 
    \[
    x\mapsto \int_{\R^m}f(x,y)\dlat y
    \]
    is finite on $\R^n$, it is log-concave on $\R^n$.
\end{prop}
Our two-dimensional inequality is as follows; it yields a new inequality for second derivatives of log-concave functions integrated against radial weights, which may be of independent interest. We use subscripts to denote partial derivatives in a given variable, e.g.,
\[
f_r(r,s)=\frac{\partial}{\partial r}f(r,s) \qquad \text{and} \quad f_{rs}=\frac{\partial^2}{\partial r\partial s}f(r,s).
\]
We write $\nabla f$ and $\nabla^2 f=(\partial_{ij}f)_{i,j}$ for the gradient and Hessian, respectively, and use $\|\cdot\|$ for the operator norm on matrices.
\begin{thm}
\label{t:theorem}
    Let $f:\R_+^2\to \R_+$ be a function in the variables $(r,s)$ such that
    \begin{enumerate}
        \item $\max f=f(0,0)$,
        \item $f$ is twice continuously differentiable and, for some $A,B>0$,
\[
|f(r,s)|+|\nabla f(r,s)|+\|\nabla^2 f(r,s)\|
\leq Ae^{-B\sqrt{r^2+s^2}}
\]
for all $(r,s)\in\R_+^2$.
        
        \item $f\in \operatorname{LC}_2$ (when extended by zero to the rest of $\R^2$).
    \end{enumerate}
    Then, for every $p>-1, p\neq 0$, we have
    \begin{equation}
    \label{eq:final_goal_updated_2d}
     \left(\int_{\R_+}\!\!\!r^{p+1} f_{rr}(r,0)\dlat r\right)\left(\int_{\R_+}\!\!\!r^{p+1}f_{ss}(r,0) \dlat r\right) \leq \left(\int_{\R_+}\!\!\!r^{p+1}f_{rs}(r,0)\dlat r\right)^2.
    \end{equation}
\end{thm}

As a principal application, we resolve a longstanding question of Richard Gardner and Gaoyong Zhang concerning the convexity of radial $p$th mean bodies. In the seminal work \cite{GZ98},  R. Gardner and G. Zhang introduced the radial $p$th mean bodies $R_p K$ of a convex body  $K\subset \R^n$, where $p>-1$. They form a family of convex bodies interpolating between the difference body and the polar projection body. Radial $p$th mean bodies have since been generalized in several directions; see \cite{LRZ22,LP25,HLPRY25,XGL14,LX24,LMU25}. Nevertheless, their convexity for $p\in(-1,0)$ has remained open for nearly three decades.

\begin{ques}
\label{q:the_question}
    Let $K\subset \R^n$ be a convex body. For $p\geq 0$, the radial $p$th mean bodies $R_p K$ are convex \cite[Theorem 4.3]{GZ98}. Is the same true for $p\in (-1,0)$?
\end{ques}
 This is trivially true when $n=1$, and was only recently shown to be true for $n=2$ by J. Haddad \cite{JH26}. We answer Question~\ref{q:the_question} in the affirmative in all dimensions. Our approach is fundamentally different than that in \cite{JH26}.

\begin{thm}
\label{t:main}
    Let $K\subset \R^n$ be a convex body. Then, for every $p>-1$, $R_p K$ is a convex body symmetric about the origin.
\end{thm}

A primary motivation for the study of radial $p$th mean bodies is the fact that as $p\to +\infty$, the bodies $R_p K$ converge (in, say, the Hausdorff metric) to $$DK=\{x\in \R^n: K\cap (K+x)\neq \emptyset\},$$ which is the \textit{difference body} of $K$, and as  $p\to (-1)^+,$ the dilated bodies $(1+p)^\frac{1}{p}\vol_n(K)^{-1}R_p K$ converge to $$\Pi^\circ K =\left\{x\in  \R^n:|x|\vol_{n-1}\left(P_{\left(\frac{x}{|x|}\right)^\perp}K\right) \leq 1\right\},$$ the \textit{polar projection body} of $K$. Here, we denote by $\vol_k$ the $k$-dimensional Lebesgue measure and $P_{\theta^\perp} K$ denotes the orthogonal projection of $K$ onto $\theta^\perp=\left\{x\in\R^n:\langle x,\theta\rangle=0\right\}$ for $\theta\in\s^{n-1}$. In turn, $\s^{n-1}$ is the unit Euclidean sphere in $\R^n$.

The radial $p$th mean bodies were shown to satisfy reverse affine isoperimetric inequalities, with extremizers being simplexes \cite[Theorem 5.5]{GZ98}. As a consequence of the above limiting procedures, these inequalities include the renowned Rogers-Shephard inequality \cite{RS57} for $DK$ and Zhang's projection inequality \cite{GZ91_2} for $\Pi^\circ K$. More recently, J. Haddad and M. Ludwig \cite{HL25,HL24} established affine isoperimetric-type inequalities for $R_p K$, with extremizers being ellipsoids. At the limit, they recover the Petty Projection inequality \cite{CMP85} for $\Pi^\circ K$. 

This paper is organized as follows. In Section~\ref{sec:prelim}, we introduce Ball bodies and their properties.  In Section~\ref{sec:theorem}, we prove the two-dimensional inequality in Theorem~\ref{t:theorem}. In Section~\ref{sec:ball_expand}, we prove Theorem~\ref{t:ball_expanded}. Finally, in Section~\ref{sec:main}, we deduce Theorem~\ref{t:main} from Theorem~\ref{t:ball_expanded} and discuss its applications.

\section{Preliminaries on K. Ball Bodies}
\label{sec:prelim}
We need some classical properties. The first is the following well-known characterization of integrable log-concave functions (see, e.g. \cite{BK07}).
\begin{prop}
\label{p:integrable}
    Let $f:\R^n\rightarrow \R_+$ be an upper semi-continuous, log-concave function whose support has non-empty interior. The following properties are equivalent:
    \begin{enumerate}
        \item[(i)] Integrability: $f \in \LC$,
        \item[(ii)] There exist constants $A,B>0$ such that $$f(x)\le A\, e^{-B\, |x|}, \quad x\in \R^n.$$
    \end{enumerate}
    In particular, $f$ has finite moments of all orders.
\end{prop}

We will use Proposition~\ref{p:integrable} repeatedly, often without explicit reference; for example, when performing integration by parts. The second classical fact we need concerns approximation of finite convex functions (see, e.g., \cite[Theorem 10.8]{RTR70}).
\begin{prop}
\label{p:convex_converge}
Let $V_j:\R^n\to\R$ be a sequence of convex functions converging pointwise to a finite function $V_\infty:\R^n\to\R$. Then $V_\infty$ is convex, and the convergence is uniform on compact subsets of $\R^n$.
\end{prop}

We now turn to our extension of K. Ball's bodies. We will need some definitions.
\begin{defn}
\label{def:star}
    A set $L\subset \R^n$ is a \textit{star-shaped set} (with respect to the origin) if $o\in L$ and $[o,x]\subset L$ for all $x\in L$. Furthermore, $L$ is a \textit{star body} if it is a star-shaped set that is compact with non-empty interior, and if its Minkowski functional, given by
    \[
\|x\|_L=\inf\{t>0:x\in tL\},
    \]
    is continuous on $\R^n$.
\end{defn}\noindent
Every convex body containing the origin in its interior is a star body. A priori, a star body $L$ is convex if and only if its Minkowski functional is convex. We will explicitly use the terminology \textit{gauge} for convex Minkowski functionals. With these definitions available, we recall a classical result by K. Ball for the class $\LC^0$ introduced above.

\begin{prop}[Theorem 5 in \cite{Ball88}]
\label{p:radial_ball}
   If $f \in \LC^0$ then, for every $p> 0,$ the function on $\R^{n}$ given by 
    \begin{equation}
    \label{eq:keith_ball_body}
   x\mapsto\left(\frac{p}{f(o)}\int_0^\infty f(rx)r^{p-1}\dlat r\right)^{-\frac{1}{p}}\end{equation}
    is a gauge, whose unit ball $K_p(f)$ is a convex body containing the origin in its interior.
\end{prop}
Theorem~\ref{t:ball_expanded} extends Proposition~\ref{p:radial_ball} to $p\in(-1,0]$. The more general case of functions $g\in\LC$ with $\max g=g(o)$ and $o\in\partial\supp(g)$ will be treated in a sequel. We remark that if one restricts the range of $p$, it was shown by S. Bobkov \cite{LBobkov} that Proposition~\ref{p:radial_ball} holds for a larger class of functions, the so-called $s$-concave functions, $s<0$.

An alternative reading of \eqref{eq:keith_ball_body} is that, for every $p >0$, the mapping $f\mapsto K_p(f)$ assigns a convex body to each $f\in\LC^0$. Motivated by investigations in \cite{FLM20,KPY08}, we define the following extension of K. Ball's bodies to $p\in(-1,0]$, together with the limiting case $p=\infty$, and show that they are star bodies for finite $p>-1$. This extension was developed simultaneously in the present work and in the joint work of the author \cite{LMU25}. The joint work focuses on reverse isoperimetric-type inequalities for $K_p(g)$; the two works should be considered complementary. 

We recall that given a monotonically increasing left-continuous function $h$ on an interval $(a,b)$, we can define its Lebesgue-Stieltjes measure $\dlat \,h$.
\begin{defn}
\label{def:keith_ball_bodies}
Let $g\in\LC^0$. Then, its $p$th Ball body is the star-shaped set $K_p(g)$ whose Minkowski functional is given by $\|o\|_{K_p(g)}=0$, and, for $x\in\R^n\setminus\{o\},$
\begin{equation}
    \label{eq:kbb}
    \begin{split}
    \|x\|_{K_p(g)}=\begin{cases}
    \left(\frac{p}{g(o)}\int_0^\infty r^{p-1}(g(rx)-g(o))\; \dlat r\right)^{-\frac{1}{p}}, & p\in (-1,0),
    \\
    \exp\left(-\frac{1}{g(o)}\int_0^\infty\log(r)\; \dlat \left(-g(rx)\right)\right), &p=0,
    \\
    \left(\frac{p}{g(o)}\int_0^\infty g(rx)r^{p-1}\, \dlat r\right)^{-\frac{1}{p}}, &p>0,
    \\
    \|x\|_{\supp(g)}, & p=\infty.
    \end{cases}
    \end{split}
\end{equation}
\end{defn}
We discuss momentarily the measure $\dlat\bigl(-g(rx)\bigr)$. Notice
that the function $r\mapsto g(rx)$ is supported on $[0,\|x\|_{\supp(g)}^{-1}]$. We have that
\[
\dlat\bigl(-g(rx)\bigr)
=
-\chi_{(0,\|x\|_{\supp(g)}^{-1})}(r)
\frac{\partial}{\partial r}g(rx)\,\dlat r
+
a_x\delta_{\|x\|_{\supp(g)}^{-1}}(r),
\]
where $\frac{\partial}{\partial r}$ denotes the right-sided derivative of $g(rx)$ in $r$, which exists on $(0,\|x\|_{\supp(g)}^{-1})$ since the convex function $r\mapsto-\log g(rx)$ has right-sided derivatives everywhere on its domain, and $$a_x=\lim_{r\to \left(\|x\|_{\supp(g)}^{-1}\right)^{-}}g(rx).$$ If $\|x\|_{\supp(g)}^{-1}=\infty$, this atom is omitted. The total mass of the measure is $g(o)$.

Notice that we can unify the cases $p>0$ and $p\in (-1,0)$ of the definition of $K_p(g)$ via integration by parts: $\|o\|_{K_p(g)}=0$ and, for $x\in \R^n\setminus\{o\}$,
\begin{equation}
    \|x\|_{K_p(g)}=\left(\frac{1}{g(o)}\int_0^\infty r^p\dlat\left(-g(rx)\right)\right)^{-\frac{1}{p}}, \qquad p>-1,\; p\neq 0,
    \label{eq:unified}
\end{equation}
The logarithmic formula at $p=0$ is the limit of \eqref{eq:unified}. Moreover, the sets $K_p(g)$ are star bodies for $-1<p<\infty$.

\begin{lem}
\label{l:star_body}
Let $g\in\LC^0$ and $-1<p<\infty$. Then $\|\cdot\|_{K_p(g)}$ is finite and continuous on $\R^n$, and is positive on $\R^n\setminus\{o\}$. Consequently, $K_p(g)$ is a star body containing the origin in its interior.
\end{lem}
\begin{proof}
    Notice first that the positivity and local Lipschitz continuity of $g$ near the origin, together with the exponential bound from Proposition~\ref{p:integrable}, give $0<\|\theta\|_{K_p(g)}<\infty$ for every $\theta\in\s^{n-1}$ and $p\neq0$. For $p=0$, Lebesgue-Stieltjes integration by parts gives
\[
-\log\|\theta\|_{K_0(g)}
=
\int_0^1
\left(\frac{g(r\theta)}{g(o)}-1\right)\frac{\dlat r}{r}
+
\int_1^\infty
\frac{g(r\theta)}{g(o)}\frac{\dlat r}{r},
\]
and both integrals are finite by the same bounds. For $\theta_j\to\theta$, we have $g(r\theta_j)\to g(r\theta)$ for almost every $r>0$. The preceding bounds, uniformly in $\theta$, therefore give continuity of $\|\cdot\|_{K_p(g)}$ on $\s^{n-1}$ by dominated convergence. Consequently, $\|\cdot\|_{K_p(g)}$ has a positive minimum and a finite maximum on $\s^{n-1}$, which, by homogeneity, establishes the claim.
\end{proof}

Finally, we have the following proposition, which concerns the limiting set as $p\to -1$. In this instance, we use $\frac{\partial}{\partial r}\big|_{r=0^+}$ to denote the right-hand derivative operator of a function in the variable $r$ at zero. It was shown in \cite[Proposition 3.10]{LMU25} that one has the pointwise limit
\begin{equation}
\label{eq:LMU_limit}
\lim_{p\to(-1)^+}(1+p)^{-\frac{1}{p}}
\|x\|_{K_p(g)}
=
-\frac{\partial}{\partial r}\bigg|_{r=0^+}
\frac{g(rx)}{g(o)}, \qquad x\in \R^{n}.
\end{equation}
By Theorem~\ref{t:ball_expanded}, $(1+p)^{-\frac{1}{p}}\|\cdot\|_{K_p(g)}$ is convex on $\R^n$ for $-1<p<0$. Once one establishes that the right-hand side of \eqref{eq:LMU_limit} is finite, one deduces from Proposition~\ref{p:convex_converge} that the pointwise limit is also convex. We provide here a self-contained proof of this convexity, without relying on knowledge of \eqref{eq:LMU_limit}.
\begin{prop}
\label{p:limiting_body}
    Let $g\in \LC^0$. Then, the function
    \[
    x\longmapsto -\frac{\partial}{\partial r}\bigg|_{r=0^+}\frac{g(rx)}{g(o)}, \qquad x\in  \R^{n},
    \]
    is finite, nonnegative, convex, and positively $1$-homogeneous. Consequently,
\[
\widetilde K_{-1}(g):=\left\{x\in \R^n:-\frac{\partial}{\partial r}\bigg|_{r=0^+}\frac{g(rx)}{g(o)}\le1\right\}
\]
is a closed convex set containing the origin in its interior.
\end{prop}
\begin{proof}
Define the functions
\[
V(x)=-\log\frac{g(x)}{g(o)},\qquad
F(x)=-\left.\frac{\partial}{\partial r}\frac{g(rx)}{g(o)}\right|_{r=0^+}, \quad x\in\R^n.
\]
Observe that $V$ is convex and finite in a neighborhood of the origin; therefore, it is locally Lipschitz there. Thus, its right-sided directional derivatives at the origin, given by $F$, exist and are finite. Since $V(o)=0$, we have from the convexity of $V$, for every $x,y\in\R^n$, $\lambda\in(0,1)$ and sufficiently small $r>0$, that
\[
\frac{V(r(\lambda x+(1-\lambda)y))}{r}
\le
\lambda\frac{V(rx)}{r}
+(1-\lambda)\frac{V(ry)}{r}.
\]
Sending $r\to 0^+$, we deduce
\[
F(\lambda x+(1-\lambda)y)
\le
\lambda F(x)
+(1-\lambda) F(y).
\]

 The nonnegativity and positive $1$-homogeneity of $F$ are immediate. Since this finite convex function is continuous and vanishes at the origin, its unit ball is closed and convex and contains the origin in its interior.
\end{proof}

Geometrically, if $h_L(u)=\sup_{y\in L}\langle y,u\rangle$ is the support function of a convex set $L\subset \R^n$, and if $$\partial V(x)=\left\{z\in \R^n: V(y) \geq V(x)+ \langle z,y-x \rangle \, \forall y \in \operatorname{dom}(V)\right\}, \, x\in \operatorname{dom}(V),$$ is the subdifferential of a convex function $V:\R^n\to (-\infty,\infty]$ with non-empty domain, then $F=h_{\partial V(o)}$ and $\tilde K_{-1}(g)=\left(\partial V(o)\right)^\circ$, where $F$ and $V$ are as in the proof of Proposition~\ref{p:limiting_body}, and we defined the polar of a convex set $L\subset \R^n$ containing the origin via $\|\cdot \|_{L^\circ}=h_L$.

 \section{A two-dimensional integral inequality}
    \label{sec:theorem}
    In this section, we prove Theorem~\ref{t:theorem}. We will use the notation, for $M$ a symmetric matrix $n\times n$, $M \preceq 0$ if $\langle Mx,x\rangle \leq 0$ for all $x\in \R^n$. We write $M \succeq 0$ if $-M \preceq 0$. Finally, we write $M\preceq L$ if $M-L \preceq 0$.
    
    Next, we introduce the function

\[
\Phi(a,b):=\int_{0}^{\infty} r^{p+1}\,f(r+a,b)\,\dlat r,
\qquad (a,b)\in\R^2.
\]
In what follows, all derivatives at points of the boundary of
$\R_+^2$, and in particular at $(0,0)$, are understood as derivatives
taken from within $\R_+^2$. Since $f\in \mathrm{LC}_2$ is integrable, it admits exponential decay at infinity; by assumption, its partial derivatives are also bounded by an exponential. Hence, differentiation under the integral sign is justified by dominated convergence, and we deduce that the derivatives of $\Phi$ are, for $(a,b)\in \R_+^2$,
\begin{equation}
\label{eq:deriv_a}
\Phi_a(a,b)=\frac{\partial}{\partial a}\Phi(a,b)
=\int_{0}^{\infty} r^{p+1}\,f_r(r+a,b)\,\dlat r,
\end{equation}
\begin{equation}
\label{eq:deriv_b}
\Phi_b(a,b)=\frac{\partial}{\partial b}\Phi(a,b)
=\int_{0}^{\infty} r^{p+1}\,f_s(r+a,b)\,\dlat r,
\end{equation}
and
\[
\Phi_{aa}(a,b)=\frac{\partial^2}{\partial a^2}\Phi(a,b)
=\int_{0}^{\infty} r^{p+1}\,f_{rr}(r+a,b)\,\dlat r,
\]
\[
\Phi_{ab}(a,b)=\frac{\partial^2}{\partial a\,\partial b}\Phi(a,b)
=\int_{0}^{\infty} r^{p+1}\,f_{rs}(r+a,b)\,\dlat r,
\]
\[
\Phi_{bb}(a,b)=\frac{\partial^2}{\partial b^2}\Phi(a,b)
=\int_{0}^{\infty} r^{p+1}\,f_{ss}(r+a,b)\,\dlat r.
\]
Its gradient and Hessian are
\[
\nabla\Phi(a,b)=
\begin{pmatrix}
\Phi_a(a,b)\\[1mm]
\Phi_b(a,b)
\end{pmatrix},
\qquad
\nabla^2\Phi(a,b)=
\begin{pmatrix}
\Phi_{aa}(a,b) & \Phi_{ab}(a,b)\\[1mm]
\Phi_{ab}(a,b) & \Phi_{bb}(a,b)
\end{pmatrix}.
\]
Thus, at $(a,b)=(0,0)$, the Hessian entries become:
\[
\nabla^2\Phi(0,0)=
\begin{pmatrix}
\displaystyle\int_{0}^{\infty} r^{p+1}f_{rr}(r,0)\,\dlat r
&
\displaystyle\int_{0}^{\infty} r^{p+1}f_{rs}(r,0)\,\dlat r
\\[3mm]
\displaystyle\int_{0}^{\infty} r^{p+1}f_{rs}(r,0)\,\dlat r
&
\displaystyle\int_{0}^{\infty} r^{p+1}f_{ss}(r,0)\,\dlat r
\end{pmatrix}.
\]
Consequently, its determinant is
\begin{equation}
\label{eq:det_formula}
\begin{split}
\det\nabla^2\Phi(0,0)=
\left(\int_{0}^{\infty} r^{p+1}f_{rr}(r,0)\,\dlat r\right)
&\left(\int_{0}^{\infty} r^{p+1}f_{ss}(r,0)\,\dlat r\right)
\\
&-\left(\int_{0}^{\infty} r^{p+1}f_{rs}(r,0)\,\dlat r\right)^2.
\end{split}
\end{equation}
Therefore, \eqref{eq:final_goal_updated_2d} is precisely the claim that
\begin{equation}
\label{eq:Phi_determinant}
\det\nabla^2\Phi(0,0) \leq 0.
\end{equation}

We consider two sub-cases.

{\bf Case 1.} Suppose $f_r(r,0)=0$ for almost all $r$ (which corresponds to $f(r,0)$ being constant along the ray). We compute and see that
\begin{equation}
\label{eq:aa}
\Phi_{aa}(0,0)=\int_{0}^{\infty} r^{p+1}f_{rr}(r,0)\,\dlat r
=-(p+1)\int_{0}^{\infty} r^{p}f_{r}(r,0)\,\dlat r.
\end{equation}
Thus, $\Phi_{aa}(0,0)=0$ and by \eqref{eq:det_formula}, $\det \nabla^2\Phi(0,0) \leq 0$. 

{\bf Case 2.} The other case is that $f_r(r,0)\neq0$ on a set of positive measure. To show \eqref{eq:Phi_determinant}, it suffices to show that the symmetric matrix $A:=\nabla^2\Phi(0,0)$ has at least one positive eigenvalue and at most one positive eigenvalue. Indeed, then its eigenvalues $(\lambda_1,\lambda_2)$ (which are real from the symmetry of $A$) satisfy $\lambda_1\ge 0\ge \lambda_2$, hence $\det A=\lambda_1\lambda_2\le 0$.

To this end, we need a criterion to determine the sign of the eigenvalues. Recall the following two facts for a $2\times 2$ symmetric matrix $M$ with eigenvalues $\lambda_1,\lambda_2\in \R$:
\begin{enumerate}
    \item $M$ is \textit{negative semi-definite}, i.e. $M\preceq 0$, if and only if $\lambda_1,\lambda_2\leq 0$;
    \item $M$ is \textit{positive semi-definite}, i.e. $M\succeq 0$, if and only if $\lambda_1,\lambda_2\geq 0$.
\end{enumerate}

We briefly recall why this is the case. Since $M$ is symmetric, there exists an orthogonal matrix $U$ and a diagonal matrix $\Lambda$ (with entries $\lambda_1,\lambda_2$) such that $M=U\Lambda U^T$. Thus, for every $x\in \R^2$:
\[
\langle Mx,x\rangle = \langle U\Lambda U^T x,x \rangle =  \langle \Lambda \left(U^T x\right), \left(U^Tx\right) \rangle \overset{y=U^T x}{=} \langle \Lambda y,y \rangle = \lambda_1 y_1^2 + \lambda_2 y_2^2.
\]
Therefore, by taking the contrapositive of these statements, we have our criterion:
\begin{enumerate}
    \item $M$ has at least one positive eigenvalue if and only if there exists $x\in\R^2$ such that $\langle Mx,x \rangle > 0$;
    \item $M$ has at least one negative eigenvalue if and only if there exists $x\in\R^2$ such that $\langle Mx,x \rangle < 0$.
\end{enumerate}

We first show that $A$ has at least one positive eigenvalue. This will follow from the fact that $f$ has maximum at the origin and is unimodal. Indeed, we have $f_r(r,0) \leq 0$ for almost all $r$ and $f_r(r,0) < 0$ on a set of positive measure. By \eqref{eq:aa}, we deduce:
\begin{equation}\label{eq:Phi_aa_nonneg}
\Phi_{aa}(0,0)> 0.
\end{equation}
Consequently, \[
\langle \nabla^2\Phi(0,0)e_1,e_1\rangle=\Phi_{aa}(0,0) > 0,
\]
i.e. by our criterion, $\nabla^2\Phi(0,0)$ has at least one positive eigenvalue.

\medskip

Next, recall that $\nabla\Phi=\Phi\cdot\nabla\log\Phi. $ Differentiating once more yields the matrix identity
\begin{equation}\label{eq:hessian_identity}
\nabla^2\Phi
=
\Phi\Big(\nabla^2\log\Phi+\nabla\log\Phi\otimes\nabla\log\Phi\Big).
\end{equation}
Since $r\mapsto r^{p+1}$ is log-concave on $(0,\infty)$ for $p>-1$ and $(r,s)\mapsto f(r,s)$ is log-concave,
an application of Pr\'ekopa's theorem, Proposition~\ref{p:pre}, to the function $$(a,b,r)\mapsto\chi_{[0,\infty)}(r)r^{p+1}f(r+a,b)$$
implies that $(a,b)\mapsto \Phi(a,b)$ is log-concave. Hence,
\begin{equation}
\label{eq:using_log_concavity}
\nabla^2\log\Phi(a,b)\preceq 0
\end{equation}
at every $(a,b)\in(0,\infty)^2$ for which $\Phi(a,b)>0$. Therefore, at such points,
\begin{equation}\label{eq:HessPhi_rankone_bound}
\nabla^2\Phi(a,b)\ \preceq\ \Phi(a,b)\nabla\log\Phi(a,b)\otimes\nabla\log\Phi(a,b)
\end{equation}
Observe that, since $\Phi(0,0)>0$, $\Phi$ is positive in a neighborhood of $(0,0)$ relative to $\R_+^2$. From continuity, we may then send $(a,b)\to (0,0)$ and deduce from \eqref{eq:HessPhi_rankone_bound} that, for every $v\in\R^2$,
\begin{equation}
\label{eq:inner_compare}
\begin{split}
\langle Av,\,v\rangle &\leq \Phi(0,0)\langle \big(\nabla\log\Phi(0,0)\otimes \nabla \log\Phi(0,0)\big)v,v\rangle 
\\
&= \Phi(0,0)\langle \nabla \log\Phi(0,0),v \rangle^2.
\end{split}
\end{equation}
Observe that $\nabla\log\Phi(0,0)$ is a vector in $\R^2$. If $\nabla \log\Phi(0,0)=o$, we have from \eqref{eq:inner_compare} the inequality
\[
\langle Av,v\rangle\le 0 \qquad \text{for all }v\in\R^2,
\]
so $A$ is negative semi-definite, contradicting \eqref{eq:Phi_aa_nonneg}. Consequently, the vector $\nabla\log\Phi(0,0)$ cannot be the origin, and we may find a non-zero $\ell\in \R^2$ such that $$\nabla\log\Phi(0,0)\perp \ell.$$ By \eqref{eq:inner_compare},
\begin{equation}
\label{eq:bad_ell}
\langle A\ell,\,\ell\rangle \leq 0.
\end{equation}
This shows $A$ has \emph{at most one positive eigenvalue}. Indeed, suppose that both eigenvalues of $A$ were strictly positive. Since $A$ is symmetric, this would imply that $A$ is positive definite, i.e. $\langle Av,v\rangle>0$ for every $v\neq o,$ contradicting \eqref{eq:bad_ell}.

\medskip

Combining these two facts, the eigenvalues of $A=\nabla^2\Phi(0,0)$ are of the form $(\lambda_1,\lambda_2)$ with $\lambda_1\ge 0\geq \lambda_2$,
and therefore $$\det \nabla^2\Phi(0,0)=\lambda_1\lambda_2\leq 0,$$ which is exactly \eqref{eq:Phi_determinant}, and hence equivalent to \eqref{eq:final_goal_updated_2d}. We conclude. \hfill \qedsymbol

\section{Convexity of Ball bodies}
\label{sec:ball_expand}
This section is dedicated to the proof of Theorem~\ref{t:ball_expanded}. We will show that the convexity of $K_p(g)$ follows from Theorem~\ref{t:theorem}. We break the proof into two steps. 

\paragraph{\bf Step 1: The Smooth Case.} 

Throughout this section, $g\in\LC^0$. We first prove the theorem for $p\neq0$ under the additional assumption that $g$ is $C^2$ smooth with the exponential derivative bounds stated below. We then remove the regularity assumption and obtain the case $p=0$ by continuity.

First, we need a criterion for the convexity of Minkowski functionals.
\begin{prop}
\label{p:deriv_test}
Let $\|\cdot\|:\R^n\to\R_+$ be a positively $1$-homogeneous, non-negative and non-identically zero function that is $C^2$ on $\R^n\setminus\{o\}$. Then $\|\cdot\|$ is a gauge if and only if, for all
$u,\theta\in \R^n\setminus\{0\}$,
\[
\frac{\dlat^2}{\dlat t^2}\Bigg|_{t=0}\|u+t\theta\| \ge 0.
\]
\end{prop}
This reduction is classical; the fact that it suffices to check just $t=0$ is due to the arbitrariness of $u$ and $\theta$. We next have the following rudimentary identity by direct computation. We choose to denote derivatives in $t$ using Newton's dot notation.
\begin{prop}
\label{p:chain_rule_final}
Let $\|\cdot\|:\R^n\to\R_+$ be a positively $1$-homogeneous and non-identically zero function that is positive and $C^2$ on $\R^n\setminus\{o\}$. Let $p>-1$, $p\neq0,\infty$. Let
$u,\theta\in\R^n\setminus\{o\}$ and set
\[
u(t)=u+t\theta,
\qquad
H(t)=\|u(t)\|^{-p}.
\]
Then, for every $t\in\R$ such that $u(t)\neq o$,
\[
\frac{p^2}{\|u(t)\|}
\frac{\dlat^2}{\dlat t^2}\|u(t)\|=(1+p)\left(\frac{\dot H(t)}{H(t)}\right)^2-p\left(\frac{\ddot H(t)}{H(t)}\right).
\]
\end{prop}

We now specialize the above reductions to $K_p(g)$ for $g\in \LC^0$.
\begin{thm}
\label{t:main_argument_ball}
Let $g\in\LC^0$ be twice continuously differentiable and assume
that there exist constants $A,B>0$ such that
\begin{equation}
\label{eq:smooth_hypo}
|g(x)|+|\nabla g(x)|+\|\nabla^2 g(x)\|
\leq Ae^{-B|x|},
\qquad x\in\R^n.
\end{equation}

Let $-1<p<\infty$, $p\neq0$, and fix
$u,\theta\in\R^n\setminus\{o\}$.
Set $u(t)=u+t\theta$ and define
$H(t)=\|u(t)\|_{K_p(g)}^{-p}$ for $t$ in a neighborhood
of zero on which $u(t)\neq o$.
Then $H$ is twice continuously differentiable there.
Writing $H=H(0)$, $\dot H=\dot H(0)$, and
$\ddot H=\ddot H(0)$, we have
\begin{equation}
\label{eq:inequalities}
p\left(\frac{\ddot H}{H}\right)
\leq
(1+p)\left(\frac{\dot H}{H}\right)^2.
\end{equation}
In particular, $\|\cdot\|_{K_p(g)}$ is convex on $\R^n$.
\end{thm}

\begin{proof}
By the definition of $K_p(g)$, for $t$ sufficiently close to zero,
\begin{equation}
\label{eq:the_function}
H(t)\!=\!\|u(t)\|_{K_p(g)}^{-p}\!=\!
\begin{cases}
\frac{p}{g(o)}
\int_{\R_+}r^{p-1}g(ru(t))\,\dlat r,
& p>0,
\\
\frac{p}{g(o)}
\int_{\R_+}r^{p-1}\left(g(ru(t))-g(o)\right)\dlat r,
& p\in(-1,0).
\end{cases}
\end{equation}

If $u$ and $\theta$ are linearly dependent, then the desired
inequality follows immediately from the positive $1$-homogeneity
of $\|\cdot\|_{K_p(g)}$.
We may therefore assume that $u$ and $\theta$ are linearly
independent.

We reduce \eqref{eq:inequalities} to a two-dimensional inequality.
To this end, define
\[
G(r,s)=g(ru+s\theta),
\qquad (r,s)\in\R^2.
\]
We use subscripts to denote partial derivatives, so that
\begin{equation}
\label{eq:G_dervis}
\begin{split}
G_s(r,0)
&=\frac{\partial G}{\partial s}(r,0)
=\langle\nabla g(ru),\theta\rangle,
\\
G_r(r,0)
&=\frac{\partial G}{\partial r}(r,0)
=\langle\nabla g(ru),u\rangle
=\frac{\partial}{\partial r}g(ru).
\end{split}
\end{equation}

Choose $\eta>0$ such that $|u+t\theta|\ge |u|/2$
whenever $|t|\le\eta$, and restrict attention to
$t\in(-\eta,\eta)$.
For $p>-1$, $p\neq0$, subtracting $H=H(0)$ from $H(t)$
in \eqref{eq:the_function} gives
\begin{equation}
\label{eq:the_function_updated}
H(t)=\frac{p}{g(o)}
\int_{\R_+}r^{p-1}\left(g(ru(t))-g(ru)\right)\dlat r+H.
\end{equation}
Here $g(ru(t))-g(ru)$ is integrable in $r$ against $r^{p-1}$. Indeed, by the chain rule,
\[
\frac{\partial}{\partial t}g(r(u+t\theta))
=
r\langle\nabla g(r(u+t\theta)),\theta\rangle.
\]
By Cauchy--Schwarz and \eqref{eq:smooth_hypo},
\begin{equation}
\label{eq:using_decay}
\begin{split}
\left|\frac{\partial}{\partial t}g(r(u+t\theta))\right|
&\le r|\theta|\,|\nabla g(r(u+t\theta))|
\\
&\le A|\theta| r e^{-Br|u+t\theta|}
\\
&\le A|\theta| r e^{-\frac{B|u|}{2}r}.
\end{split}
\end{equation}
Thus, with $C=A|\theta|$ and $c=B|u|/2$,
the mean value theorem gives
\begin{equation}
\label{eq:lipschitz_t}
|g(ru(t+h))-g(ru(t))|
\le C|h|r e^{-cr},
\qquad r>0,\quad t,t+h\in(-\eta,\eta).
\end{equation}
Taking $t=0$ proves the asserted integrability, since $r^p e^{-cr}$ is integrable on $\R_+$ for $p>-1$. Moreover, for each fixed $r>0$, the chain rule gives
\begin{equation}
\label{eq:g_difference_r}
\lim_{h\to0}
r^{p-1}\frac{g(ru(t+h))-g(ru(t))}{h}
=
r^p\langle\nabla g(ru(t)),\theta\rangle, \quad t\in (-\eta,\eta).
\end{equation}
Since $Cr^p e^{-cr}$ is integrable on $\R_+$ for $p>-1$, \eqref{eq:lipschitz_t} allows us to apply the dominated convergence theorem to the difference quotients of $H$. Thus $H$ is differentiable, with derivative, by \eqref{eq:g_difference_r},
\begin{equation}
\label{eq:the_deriv}
g(o)\cdot\dot H(t)
=
p\int_{\R_+}r^p
\langle\nabla g(ru(t)),\theta\rangle\,\dlat r.
\end{equation}

Evaluating \eqref{eq:the_deriv} at $t=0$ and using
\eqref{eq:G_dervis}, we obtain
\begin{equation}
\label{eq:the_deriv_0}
\begin{split}
g(o)\cdot\dot H
&=p\int_{\R_+}r^p
\langle\nabla g(ru),\theta\rangle\,\dlat r
\\
&=p\int_{\R_+}r^pG_s(r,0)\,\dlat r
\\
&=-\frac{p}{p+1}
\int_{\R_+}r^{p+1}G_{rs}(r,0)\,\dlat r,
\end{split}
\end{equation}
where the last equality follows from integration by parts.

We now take the second derivative of $H$. We mimic the derivation of \eqref{eq:lipschitz_t} and deduce from the chain rule and \eqref{eq:smooth_hypo} that
\[
\left|\frac{\partial^2}{\partial t^2}g(ru(t))\right|
=
r^2\left|\langle\nabla^2g(ru(t))\theta,\theta\rangle\right|
\le A|\theta|^2r^2e^{-cr},
\]
uniformly for $|t|\le\eta$.
By the mean value theorem,
\[
r^p
\left|
\frac{\langle\nabla g(ru(t+h))-\nabla g(ru(t)),\theta\rangle}{h}
\right|
\le A|\theta|^2r^{p+1}e^{-cr},
\]
whenever $r>0$, $h\neq0$, and $t,t+h\in(-\eta,\eta)$.
Since this bound is integrable for $p>-1$, the same
dominated-convergence argument shows that $H$ is twice
differentiable. Evaluating the resulting formula at $t=0$ yields
\begin{equation}
\label{eq:the_second_deriv}
\begin{split}
g(o)\cdot\ddot H
&=p\int_{\R_+}r^{p+1}
\langle\nabla^2g(ru)\theta,\theta\rangle\,\dlat r
\\
&=p\int_{\R_+}r^{p+1}G_{ss}(r,0)\,\dlat r.
\end{split}
\end{equation}
Continuity of $\nabla^2g$ and the same dominating bound
also imply continuity of $\ddot H(t)$ near zero.

To express $H$ in terms of $G_{rr}$, recall
\eqref{eq:unified}:
\begin{equation}
\label{eq:updated_unified}
H(t)=\frac{1}{g(o)}
\int_0^\infty
\left(-\frac{\partial}{\partial r}g(ru(t))\right)r^p\,\dlat r.
\end{equation}
Evaluating at $t=0$ and integrating by parts gives
\begin{equation}
\label{eq:updated_unified_0}
g(o)\cdot H
=
\int_{\R_+}r^p\left(-G_r(r,0)\right)\dlat r
=
\frac{1}{p+1}
\int_{\R_+}r^{p+1}G_{rr}(r,0)\,\dlat r.
\end{equation}

Substituting \eqref{eq:updated_unified_0},
\eqref{eq:the_deriv_0}, and \eqref{eq:the_second_deriv}
into \eqref{eq:inequalities}, we obtain the equivalent inequality
\[
\left(\int_{\R_+}r^{p+1}G_{rr}(r,0)\,\dlat r\right)
\left(\int_{\R_+}r^{p+1}G_{ss}(r,0)\,\dlat r\right)
\le
\left(\int_{\R_+}r^{p+1}G_{rs}(r,0)\,\dlat r\right)^2.
\]
Since $u$ and $\theta$ are linearly independent,
there exists $a>0$ such that
\[
|ru+s\theta|\ge a\sqrt{r^2+s^2},
\qquad (r,s)\in\R^2.
\]
Together with \eqref{eq:smooth_hypo} and the assumptions
on $g$, this shows that the restriction of $G$ to $\R_+^2$
satisfies the hypotheses of Theorem~\ref{t:theorem}.
The desired inequality follows.

The ``in particular'' part follows from
 Propositions~\ref{p:deriv_test} and \ref{p:chain_rule_final}.
\end{proof}

    \paragraph{\bf Step 2: Removing the Regularity Assumptions.} We now remove the assumptions made on $g$ and $p$. First, we show that the assumption that $g$ is $C^2$ smooth can be dropped. Temporarily, we will add the assumption that $g$ has a unique maximum, and then drop this later as well.
    
    We now show that an arbitrary $g\in \LC^0$ with a unique maximum at the origin can be approximated by a sequence of $C^\infty$ smooth $g_k\in \LC^0$ which attain their maxima at the origin.

To this end, set, for $k\in \N$
\[
\gamma_k(x)=\frac{1}{\left(\frac{4\pi}{k}\right)^\frac{n}{2}}e^{-k\frac{|x|^2}{4}},
\]
the centered Gaussian probability density with variance $\frac{2}{k}$. Then $\gamma_k$ is $C^\infty$ smooth and $\|\gamma_k\|_{L^1}= 1$. Finally, we set 
\begin{equation}
\label{eq:g_eps}
\tilde g_k(x)=(g\ast \gamma_k)(x)=\int_{\R^n}g(z)\gamma_k(x-z)\dlat z.
\end{equation}
Since $g$ and $\gamma_k$ are log-concave, we have by Pr\'ekopa's theorem, Proposition~\ref{p:pre}, that their convolution is log-concave, i.e. $\tilde g_k$ is log-concave. We also have that $\tilde g_k$ is integrable:
\[
0<\int_{\R^n}\tilde g_k(x)\dlat x = \int_{\R^n}\int_{\R^n}g(z)\gamma_k(x-z)\dlat z\,\dlat x = \|g\|_{L^1}\|\gamma_k\|_{L^1} = \|g\|_{L^1}.
\]
It is classical that $\tilde g_k\to g$ almost everywhere and uniformly on compact subsets of the interior of the support of $g$ (see \cite[Fact 2.5]{CEFL24}). 

We must handle the issue that $\tilde g_k$ may no longer have its maximum at the origin (indeed, consider the one-dimensional example $g=e^{-|\,\cdot\,|}\chi_{[-1,\infty)}$). To correct for this, choose
$x(k)\in\arg\max\tilde g_k$; such a point exists because
$\tilde g_k$ is continuous and decays exponentially by
Proposition~\ref{p:integrable}. We claim that $x(k)\to o$.

Choose $\varepsilon_0>0$ such that $\varepsilon_0 B_2^n\subset\operatorname{int}\supp(g)$. Fix any $\varepsilon\in(0,\varepsilon_0]$. Since $g$ is continuous on $\varepsilon B_2^n$ and has its unique maximum at the origin,
\[
\delta_\varepsilon
:=g(o)-\max_{|x|=\varepsilon}g(x)>0.
\]
By uniform convergence on $\varepsilon B_2^n$, there exists $N_\varepsilon\in\N$ such that, for every $k\ge N_\varepsilon$,
\[
\sup_{x\in\varepsilon B_2^n}
|\tilde g_k(x)-g(x)|<\frac{\delta_\varepsilon}{3}.
\]
Consequently, for every $k\ge N_\varepsilon$,
\begin{equation}
\label{eq:using_uniform}
\max_{|x|=\varepsilon}\tilde g_k(x)
<g(o)-\frac{2\delta_\varepsilon}{3}
<\tilde g_k(o).
\end{equation}

If $|x(k)|\ge\varepsilon$, set $\lambda=\varepsilon/|x(k)|\in(0,1]$. Since $x(k)$ is a maximizer of $\tilde g_k$, log-concavity gives
\[
\tilde g_k\left(\frac{\varepsilon x(k)}{|x(k)|}\right)
=\tilde g_k\bigl((1-\lambda)o+\lambda x(k)\bigr)
\ge \tilde g_k(o)^{1-\lambda}\tilde g_k(x(k))^\lambda
\ge \tilde g_k(o),
\]
contradicting \eqref{eq:using_uniform}. Thus, for every $\varepsilon\in(0,\varepsilon_0]$, we have $|x(k)|<\varepsilon$ whenever $k\ge N_\varepsilon$. This proves that $x(k)\to o$. We consequently define $g_k$ for $k\in\N$ via
\begin{equation}
\label{eq:g_k_def}
g_k(x)=\tilde g_k(x+x(k)).
\end{equation}
Notice that $g_k$ inherits the following properties from $\tilde g_k:$
\begin{enumerate}
    \item $g_k \in \LC^0 \quad \forall \; k\in \N,$ since $\supp (g_k) = \supp (\tilde g_k) = \R^n;$
    \item $g_k\to g$ uniformly on compact subsets of $\operatorname{int}\supp(g)$ as $k\to\infty$.
\end{enumerate}
For $z\notin\supp(g)$, choose $\delta>0$ such that $z+2\delta B_2^n$ is disjoint from $\supp(g)$. For all sufficiently large $k$, we have $|x(k)|<\delta$, and hence
\[
0\le g_k(z)\le g(o)\int_{\{|y|\ge\delta\}}\gamma_k(y)\dlat y\longrightarrow0.
\]
Thus $g_k(z)\to g(z)$ for every $z\notin\partial\supp(g)$. Since $\supp(g)$ is convex and contains the origin in its interior, each ray from the origin meets its boundary at at most one point. Consequently, for each fixed $x\neq o$, $g_k(rx)\to g(rx)$ for almost every $r>0$.
We will use these convergence properties to show that $\|\cdot\|_{K_p(g_k)}\to \|\cdot\|_{K_p(g)}$ pointwise as $k\to\infty$.
\begin{prop}
\label{p:convergence}
    Let $g\in \LC^0$ have a unique maximum at the origin and define the approximating sequence $g_k$ via \eqref{eq:g_k_def}. Fix $x\in \R^n$. Then, for every $p>-1, p\neq 0,\infty,$ it holds that $$\lim_{k\to \infty} \|x\|_{K_p(g_k)}=\|x\|_{K_p(g)}.$$
    In particular, $\|\cdot\|_{K_p(g)}$ is convex on $\R^n$.
\end{prop}
\begin{proof}
    At $x=o$, we have $\|o\|_{K_p(g_k)}=\|o\|_{K_p(g)}=0$ for all $k$. For all other $x\in \R^n$, it suffices to show that $\|x\|_{K_p(g_k)}^{-p}\to \|x\|_{K_p(g)}^{-p}$, since $t\mapsto t^{-\frac{1}{p}}$ is continuous away from zero. Recall that there exist $a,c>0$ such that  
\begin{equation}
\label{eq:g_bounds}
g(x)\leq a e^{-c|x|}.
\end{equation}
Then, for every $x\in \R^n\setminus\{o\}$, we have, for $r>0$,
\begin{align*}
g_k(rx) &= \int_{\R^n}g(z)\gamma_k(rx +x(k)-z)\dlat z = \int_{\R^n}g(rx +x(k)-y)\gamma_k(y)\dlat y.
\end{align*}
We now split the integral. For ease of presentation, we set $x^\prime=rx+x(k)$. Over $|y|\leq \frac{1}{2}|x^\prime|$ we have, by the reverse triangle inequality $|x^\prime-y| \geq |x^\prime|-|y|\geq \frac{1}{2}|x^\prime|$ and \eqref{eq:g_bounds}
\begin{align*}
\int_{\{y\in\R^n:|y|\leq \frac{1}{2}|x^\prime|\}}\!\!\!\!\!\!\!\!\!\!\!\!\!\!\!g(x^\prime-y)\gamma_k(y)\dlat y &\leq a \int_{\{y\in\R^n:|y|\leq \frac{1}{2}|x^\prime|\}}\!\!\!\!\!\!\!\!\!\!\!\!e^{-c|x^\prime-y|}\gamma_k(y)\dlat y 
\\
&\leq  ae^{-\frac{c}{2}|x^\prime|} \int_{\{y\in\R^n:|y|\leq\frac{1}{2}|x^\prime|\}}\!\!\!\!\!\!\!\!\!\!\gamma_k(y)\dlat y \leq ae^{-\frac{c}{2}|x^\prime|}.
\end{align*}

For $|y|> \frac{1}{2}|x^\prime|$, we use that $\max g=g(o)$ and then Chernoff's inequality for the standard Gaussian distribution to get 
\begin{align*}
\int_{\{y\in\R^n:|y| > \frac{1}{2}|x^\prime|\}}\!\!\!\!\!\!\!\!\!\!\!\!g(x^\prime-y)\gamma_k(y)\dlat y &\leq g(o) \int_{\{y\in\R^n:|y| > \frac{1}{2}|x^\prime|\}}\gamma_k(y)\dlat y 
\\
&\leq  2ng(o) e^{-\frac{k|x^\prime|^2}{16n}}.
\end{align*}
We then use the fact that Gaussians are integrable and log-concave to deduce the existence of $a^\prime,c^\prime>0$ independent of $k$ such that $2ng(o) e^{-\frac{k|x^\prime|^2}{16n}} \leq a^\prime e^{-\frac{c^\prime}{2}|x^\prime|}$. Then, since there exist $A,C>0$ such that
\[
ae^{-\frac c2|x'|}
+a'e^{-\frac{c'}2|x'|}
\leq
Ae^{-C|x'|},
\]
we have
\[
|g_k(rx)| \leq Ae^{-C|rx +x(k)|} \leq Ae^{C|x(k)|}e^{-Cr|x|},
\]
where, in the last line, we used the reverse triangle inequality in the form $|rx +x(k)|\geq r|x|-|x(k)|.$ We now use the fact that, since $x(k) \to o$, there exists $R>0$ large enough so that $x(k)\in RB_2^n$ for all $k$. Thus, we can find $A^\prime$ so that
\begin{equation}
\label{eq:the_bound}
g_k(rx) \leq A^\prime e^{-Cr|x|}.
\end{equation}

Since $A^\prime e^{-Cr|x|}$ is integrable against $r^{p-1}$ for $p>0$, we have, by the dominated convergence theorem, $\|x\|_{K_p(g_k)}^{-p}\to \|x\|_{K_p(g)}^{-p}$ for $p>0$. Regarding $p\in (-1,0)$, we split the integral:
\begin{equation}
\label{eq:two_ints}
\begin{split}
\|x\|_{K_p(g_k)}^{-p} = \frac{p}{g_k(o)}\int_0^{1}r^{p-1}&(g_k(rx)-g_k(o))\dlat r
\\
&+ \frac{p}{g_k(o)}\int_{1}^{\infty}r^{p-1}(g_k(rx)-g_k(o))\dlat r.
\end{split}
\end{equation}
For the latter integral in \eqref{eq:two_ints}, we have
\[
\frac{p}{g_k(o)}\int_{1}^{\infty}r^{p-1}(g_k(rx)-g_k(o))\dlat r = \frac{p}{g_k(o)}\int_{1}^{\infty}r^{p-1}g_k(rx)\dlat r + 1,
\]
and then \eqref{eq:the_bound} tells us that $\frac{p}{g_k(o)}\int_{1}^{\infty}r^{p-1}g_k(rx)\dlat r \to \frac{p}{g(o)}\int_{1}^{\infty}r^{p-1}g(rx)\dlat r$ by the dominated convergence theorem. Thus, it remains to study the first integral in \eqref{eq:two_ints}.

Therefore, we need to show that
\begin{equation}
\label{eq:p_neq_approx}
\lim_{k\to \infty} \frac{p}{g_k(o)}\int_0^{1}r^{p-1}(g_k(rx)-g_k(o))\dlat r = \frac{p}{g(o)}\int_0^{1}r^{p-1}(g(rx)-g(o))\dlat r.
\end{equation}
To handle the singularity near zero when $p\in(-1,0)$, we use log-concavity along the ray.
Fix $\eta\in (0,1)$ such that $\eta x\in \operatorname{int}\left(\supp (g)\right)$. We split the integral over $[0,\eta]$ and $(\eta,1)$. Notice that
\begin{align*}
|g_k(rx)-g_k(o)|&\leq \int_{\R^n}|g(rx+y)-g(y)|\gamma_k(x(k)-y)\dlat y 
\\
&\leq 2g(o)\int_{\R^n}\gamma_k(x(k)-y)\dlat y 
\\
&\leq 2g(o).
\end{align*}
Consequently, we may use the dominated convergence theorem and deduce
\[
\lim_{k\to \infty} \frac{p}{g_k(o)}\int_{\eta}^{1}r^{p-1}(g_k(rx)-g_k(o))\dlat r = \frac{p}{g(o)}\int_\eta^1 r^{p-1}(g(rx)-g(o))\dlat r.
\]
We now focus on the integral from $0$ to $\eta$. 

Since $g_k(o)\to g(o)>0$ and $g_k(\eta x)\to g(\eta x)>0$, there exists $c_1>0$ such that
\begin{equation}
\label{eq:g_k_lower}
\frac{g_k(\eta x)}{g_k(o)}\ge c_1>0
\end{equation}
for all $k$ large enough. For each $k$ large enough so that \eqref{eq:g_k_lower} holds, we use that the function $r\mapsto g_k(rx)$ is log-concave on $\R_+$ and achieves its maximum at $r=0$ to deduce, for every $r\in[0,\eta]$,
\[
g_k(rx) = g_k\left(\left(1-\frac{r}{\eta}\right)o+\frac{r}{\eta}(\eta x)\right)\ge g_k(o)^{1-\frac{r}{\eta}}g_k(\eta x)^\frac{r}{\eta},
\]
so, by \eqref{eq:g_k_lower},
\[
g_k(rx)\ge  g_k(o)\left(\frac{g_k(\eta x)}{g_k(o)}\right)^\frac{r}{\eta}\geq g_k(o)c_1^\frac{r}{\eta}.
\]
Therefore,
\[
0\le g_k(o)-g_k(rx)
\le g_k(o)\left(1-c_1^\frac{r}{\eta}\right).
\]
We have $1-c_1^\frac{r}{\eta} \leq Cr$ for some $C>0$ and $r\in (0,\eta)$. Thus,
\[
0\le g_k(o)-g_k(rx) \le g_k(o)Cr
\]
for some constant $C>0$ independent of $k$ and all $r\in[0,\eta]$.
Consequently, since $p<0$,
\[
\frac{p}{g_k(o)}\left(g_k(rx)-g_k(o)\right)r^{p-1}\le |p|Cr^p,
\]
and since $p>-1$, the right-hand side is integrable on $(0,\eta)$. We may again apply the dominated convergence theorem and deduce \eqref{eq:p_neq_approx}. We have thus shown that $$\|x\|_{K_p(g_k)}^{-p}\to \|x\|_{K_p(g)}^{-p}$$ for $p\in (-1,0)$ and all $x\in \R^n$.

For the ``in particular'' part, first observe that the derivatives of the Gaussian kernel $\gamma_k$ are polynomial multiples of $\gamma_k$. This, combined with Proposition~\ref{p:integrable}, shows that each $g_k$ is twice continuously differentiable and satisfies \eqref{eq:smooth_hypo}, with constants depending on $k$. Thus, by Theorem~\ref{t:main_argument_ball}, $\|\cdot\|_{K_p(g_k)}$ is convex for every $k$. By Proposition~\ref{p:convex_converge}, $\|\cdot\|_{K_p(g)}$ is therefore convex.
\end{proof}

We may now prove that $K_p(g)$ is a convex body for every $g\in \LC^0$. 

\begin{proof}[Proof of Theorem~\ref{t:ball_expanded}]
By Lemma~\ref{l:star_body}, $K_p(g)$ is a star body containing the origin in its interior for every $-1<p<\infty$. It remains to prove convexity. First fix $p>-1$, $p\neq0$. 

We have shown in Proposition~\ref{p:convergence} that $\|\cdot\|_{K_p(g)}$ is convex when $g$ attains its maximum uniquely at the origin. It remains to drop this assumption. Indeed, for arbitrary $g\in \LC^0$, set $g_j=e^{-\frac{1}{j}|\,\cdot\,|^2}g$. Then, for all $j$, $g_j$ is log-concave on $\supp (g_j)=\supp (g)$ and $\max g_j=g_j(o)=g(o)$, which is obtained \textit{uniquely} at the origin. Therefore, a simple, dominated convergence argument allows us to conclude $K_p(g_j)\to K_p(g)$ as $j\to \infty$ for all $p>-1,\,p\neq 0,\infty$. Equivalently, $\|\cdot\|_{K_p(g_j)}\to \|\cdot\|_{K_p(g)}$ pointwise. We again invoke Proposition~\ref{p:convex_converge} to deduce that $\|\cdot\|_{K_p(g)}$ is a convex function. 

Finally, we treat $p=0$ directly. Fix $x\in\R^n\setminus\{o\}$ and, for $0\le p\le1$, set
\[
J_p(x)=\frac{1}{g(o)}\left(
\int_0^1 r^{p-1}\bigl(g(rx)-g(o)\bigr)\dlat r
+\int_1^\infty r^{p-1}g(rx)\dlat r
\right).
\]
The local Lipschitz continuity of $g$ near the origin gives $g(rx)-g(o)=O(r)$ as $r\to0$, while Proposition~\ref{p:integrable} gives exponential decay at infinity. Thus, dominated convergence yields $J_p(x)\to J_0(x)$ as $p\to0^+$. By Lebesgue-Stieltjes integration by parts, as in the proof of Lemma~\ref{l:star_body},
\[
J_0(x)=\frac{1}{g(o)}\int_0^\infty\log(r)\,\dlat\bigl(-g(rx)\bigr)
=-\log\|x\|_{K_0(g)}.
\]
For $0<p\le1$, the definition of $K_p(g)$ gives
\[
\|x\|_{K_p(g)}^{-p}=1+pJ_p(x).
\]
Consequently, as $p\to0^+$,
\[
\log\|x\|_{K_p(g)}
=-\frac{\log\bigl(1+pJ_p(x)\bigr)}{p}
\longrightarrow -J_0(x)=\log\|x\|_{K_0(g)}.
\]
At $x=o$, the Minkowski functionals are all zero. Therefore, $\|\cdot\|_{K_p(g)}\to\|\cdot\|_{K_0(g)}$ pointwise as $p\to0^+$, and Proposition~\ref{p:convex_converge} shows that $\|\cdot\|_{K_0(g)}$ is convex.
\end{proof}

We mention the following probabilistic corollary of our main theorem.

\begin{cor}
\label{c:probabilistic}
Let $X$ be an $\R^n$-valued random vector whose characteristic function
\[
\phi_X(\xi)=\E e^{i\langle \xi,X\rangle}
\]
is nonnegative, integrable, and log-concave. Then, for every $s\in(0,1)$,
\[
x\mapsto \left(\E|\langle x,X\rangle|^s\right)^{1/s}
\]
is finite and defines a norm on $\R^n$.
\end{cor}

\begin{proof}
Since $\phi_X$ is continuous, $\phi_X(o)=1$, and $|\phi_X|\leq 1$, we have $\phi_X\in\LC^0$. Moreover, since $\phi_X$ is real-valued,
\[
\phi_X(-\xi)=\overline{\phi_X(\xi)}=\phi_X(\xi),
\]
so $\phi_X$ is even. Thus, by Theorem~\ref{t:ball_expanded}, $K_{-s}(\phi_X)$ is an origin-symmetric convex body.

Set
\[
c_s:=s\int_0^\infty \frac{1-\cos t}{t^{1+s}}\,\dlat t
=\Gamma(1-s)\cos\left(\frac{\pi s}{2}\right)>0.
\]
Using $\phi_X(rx)=\E\cos(r\langle x,X\rangle)$ and Tonelli's theorem, we obtain
\begin{align*}
\|x\|_{K_{-s}(\phi_X)}^s
&=s\int_0^\infty
\frac{1-\phi_X(rx)}{r^{1+s}}\,\dlat r\\
&=\E\left[
s\int_0^\infty
\frac{1-\cos(r\langle x,X\rangle)}{r^{1+s}}\,\dlat r
\right]\\
&=c_s\,\E|\langle x,X\rangle|^s.
\end{align*}
In particular, these moments are finite, and
\[
\left(\E|\langle x,X\rangle|^s\right)^{1/s}
=c_s^{-1/s}\|x\|_{K_{-s}(\phi_X)}.
\]
The right-hand side is a norm, proving the assertion.
\end{proof}
The hypotheses of Corollary~\ref{c:probabilistic} are satisfied, for example, by nondegenerate centered Gaussian distributions. The result is reminiscent of Berck's convexity theorem for $L_p$-intersection bodies \cite{BG09}. Combined with Ball's construction from Proposition~\ref{p:radial_ball}, Berck's theorem implies that, if $X$ has an even log-concave density, then $x\mapsto \left(\E|\langle x,X\rangle|^s\right)^{1/s}$ defines a norm for $0<s<1$. The two results are ultimately incomparable: the Cauchy distribution centered at the origin has a density that is not log-concave, but its characteristic function is nonnegative, integrable, and log-concave. Conversely, the uniform distribution on $[-1,1]^n$ has a log-concave density, whereas its characteristic function is neither nonnegative nor log-concave.

\section{Radial mean bodies and applications}
\label{sec:main}
We conclude by showing that Theorem~\ref{t:main} is merely a special case of Theorem~\ref{t:ball_expanded}. We start by recalling the definitions of radial $p$th mean bodies, as introduced by R. Gardner and G. Zhang \cite{GZ98}.

\begin{defn}
    Let $K\subset\R^n$ be a convex body. Then, for $p>-1$, its radial $p$th mean body $R_p K$ is the star body given by the Minkowski functional, for $x\in\R^{n}$,
    \begin{equation}
\label{eq:radial_mean_og}
    \|x\|_{R_p K}=
    \begin{cases}
        \left(\int_{K}\|x\|_{K-y}^{-p}\frac{\dlat y}{\vol_n(K)}\right)^{-\frac{1}{p}}, & p>-1,p\neq 0,
        \\
        \exp\left(\int_{K}\log\|x\|_{K-y}\frac{\dlat y}{\vol_n(K)}\right), &p=0.
    \end{cases}
\end{equation}
\end{defn}

Next, we recall that the covariogram function of a convex body $K$ is given by
\begin{equation}\label{eq:covariogram}
g_K(x)=\vol_n(K\cap (K+x)), \quad x\in\R^n;
\end{equation} 
see the recent survey by Bianchi \cite{GB23} for a rich overview of this function. We merely mention that
\begin{equation}
\label{eq:g_K_facts}
    \supp(g_K)=DK,\quad g_K\in \LC^0 \,\, \text{and} \,\, \max g_K = g_K(o)=\vol_n(K).
\end{equation}
Through an application of Fubini's theorem, one deduces an alternative representation for the Minkowski functional of $R_p K$. The identity is well-known (see \cite[Lemma 3.1]{GZ98} and also \cite{LP25} for $p\in (-1,0)$), but we will provide a proof for completeness.
\begin{prop}
\label{p:equivalence}
Let $p>-1$ and let $K\subset \R^n$ be a convex body. Then, the Minkowski functional of $R_p K$ satisfies, for $x\in\R^{n}\setminus\{o\}$:
    \begin{equation}
\|x\|_{R_p K}=\begin{cases}
    \left(p\int_0^{\infty}\left(\frac{g_K(rx)}{\vol_n(K)}\right)r^{p-1}\dlat r\right)^{-\frac{1}{p}}, & p>0,
    \\
    \exp\left(-\int_0^{\infty}\frac{\partial}{\partial r}\left(\frac{-g_K(rx)}{\vol_n(K)}\right)\log(r)\dlat r\right), & p=0,
    \\
    \left(p\int_{0}^{\infty} \left(\frac{g_K(rx)}{\vol_n(K)}-1\right) r^{p-1} \dlat r\right)^{-\frac{1}{p}}, & p\in (-1,0).
\end{cases}
\label{radial_ell}
\end{equation}
\end{prop}

\begin{proof}
By homogeneity, it suffices to establish the formulas for $\theta\in\s^{n-1}$, which we now fix. We first consider the case when $p\neq 0$. The equivalence between \eqref{eq:radial_mean_og} and \eqref{radial_ell} is an application of Fubini's theorem. Indeed, for $p>0$, we have
\begin{align*}
\int_{K} \|\theta\|_{K-x}^{-p} \dlat \, x &= p\int_K\int_0^{\|\theta\|_{K-x}^{-1}}r^{p-1}\dlat \, r\,\dlat \, x
\\
&=p \int_{0}^{\|\theta\|^{-1}_{D K}}\left(\int_{K \cap(K-r \theta)} \dlat \, x\right) r^{p-1} d r 
\\
&=p \int_{0}^{\|\theta\|_{D K}^{-1}} g_{K}(r \theta) r^{p-1} d r,
\end{align*}
where, in the second step, we used the fact that $x\in K$ and $r\theta\in K-x$ for all $0\leq r\leq \|\theta\|_{K-x}^{-1}$. Similarly, for $p\in (-1,0),$ we have 
\begin{align*}
&\int_{K} \|\theta\|_{K-x}^{-p} \dlat \, x =-p \int_K \int_{\|\theta\|_{K-x}^{-1}}^\infty  r^{p-1}\dlat \, r\dlat \, x
\\
&=-p\int_{0}^{\|\theta\|_{DK}^{-1}}\left(\int_{K\setminus {K\cap (K-r\theta)}}\dlat \, x\right)r^{p-1}\dlat \, r - p\int_K\int_{\|\theta\|_{DK}^{-1}}^\infty r^{p-1}\dlat \, r\dlat \, x.
\end{align*}
Adding and subtracting integration over $K\cap(K-r\theta),$ we obtain
\begin{align*}
\int_{K} \|\theta\|_{K-x}^{-p} \dlat \, x &=p\int_{0}^{\|\theta\|_{DK}^{-1}}(g_{K}(r\theta)-\vol_n(K))r^{p-1}\dlat \, r+\|\theta\|^{-p}_{DK}\vol_n(K)
\\
&=p\int_{0}^{\infty}(g_{K}(r\theta)-\vol_n(K))r^{p-1}\dlat \, r.
\end{align*}
By integration by parts, \eqref{radial_ell} can be rewritten, for $p\neq 0$, as
\begin{equation}\|\theta\|_{R_p K}=\left(\int_0^{\|\theta\|_{DK}^{-1}}\frac{\partial}{\partial r}\left(-\frac{g_K(r\theta)}{\vol_n(K)}\right) r^{p}\dlat \, r\right)^{-\frac{1}{p}}.
\label{eq:best_radial_form}
\end{equation}
For the final case when $p=0$, observe that the same computations give 
\[-
\int_K\log\|\theta\|_{K-y}\,\frac{\dlat y}{\vol_n(K)}
=
\int_0^1
\left(\frac{g_K(r\theta)}{\vol_n(K)}-1\right)
\frac{\dlat r}{r}+
\int_1^\infty
\frac{g_K(r\theta)}{\vol_n(K)}
\frac{\dlat r}{r}.
\]
Integration by parts then yields the second formula in \eqref{radial_ell}.
\end{proof}

The proof of the convexity of $R_p K$ is now immediate.
\begin{proof}[Proof of Theorem~\ref{t:main}]
    By Proposition~\ref{p:equivalence} and \eqref{eq:g_K_facts}, we may write
    \[
    R_p K=K_p(g_K), \qquad p>-1.
    \]
    The convexity of $R_p K$ follows from Theorem~\ref{t:ball_expanded}. Since $g_K \in \LC^0$, $R_p K$ is a convex body. Since $g_K$ is even, $R_p K$ is origin-symmetric.
\end{proof}

Theorem~\ref{t:ball_expanded} also establishes convexity for several generalizations of radial $p$th mean bodies. We describe their representations as Ball bodies associated with suitable covariograms.
\medskip

\noindent {\bf The $L^2$ Setting.} In \cite{HL26}, J. Haddad and M. Ludwig introduced what we call the $L^2$ radial $\alpha$th mean bodies $R_\alpha f$ of nonzero, nonnegative $f\in L^2(\R^n)$ for $\alpha>-1$, as part of their affine extension of an inequality by Lieb \cite{LEH83}. By defining what we call the $L^2$ covariogram of $f$,
\[
C_f(x)=\int_{\R^n}f(y)f(x+y)\dlat y,
\]
(because $C_f(o)$ is the $L^2$ norm of $f$ squared), for $f\in\LC$ the star bodies $R_\alpha f$ can be written as
\[
R_\alpha f = K_\alpha(C_f), \qquad \alpha>-1.
\]
Notice that $C_f$ is even and, when $f\in \LC$, $C_f \in \LC^0$ by Pr\'ekopa's theorem, Proposition~\ref{p:pre}. By Theorem~\ref{t:ball_expanded}, $R_\alpha f$ is an origin-symmetric convex body for all $\alpha>-1$ for $f \in \LC$.

In \cite[Section 10]{HL24}, J. Haddad and M. Ludwig considered the fractional anisotropic Sobolev energy, for $0<\alpha<1$ and $L\subset \R^n$ a star body,
\[
\mathcal{E}_\alpha(f;L)=\frac{1}{\vol_n(L)^\frac{n+\alpha}{n}}\int_{\R^n}\int_{\R^n}\frac{|f(x)-f(y)|^2}{\|x-y\|_L^{n+\alpha}}dxdy, \quad f\in L^2(\R^n),
\]
and showed, among other things, $\mathcal{E}_\alpha(f;L)\geq \mathcal{E}_\alpha(f;K_{-\alpha}(C_f))$ as an application of the dual mixed-volume inequality. Our results show, therefore, that when restricted to $f\in \LC$, this minimizer is a convex body that contains the origin in its interior.
\medskip

\noindent {\bf The $L^1$ Setting.} In \cite{ABG20}, D. Alonso-Guti\'errez, J. Bernu\'es, and B. Gonz\'alez Merino defined what we call the $L^1$ covariogram of $f\in \LC$:
\[
g_{f}(x)=\int_{\R^n}\min\{f(y),f(x+y)\}\dlat y
\]
(because $g_{f}(o)$ is the $L^1$ norm of $f$) to study a functional version of Zhang's projection inequality. 

In \cite{LMU25}, $L^1$ radial $p$th mean bodies $R_p f$ of $f\in\LC$ were defined directly via $R_pf=K_p(g_f)$. Notice $g_f\in\LC^0$ and $g_f$ is even. By Theorem~\ref{t:ball_expanded}, $R_p f$ is an origin-symmetric convex body for all $p>-1$.

\medskip

\noindent {\bf The Weighted Setting.} The weighted covariogram of a convex body $K\subset \R^n$ was introduced in \cite{LRZ22}: let $\phi\in\LC$ and let $\mu$ be the Borel measure with density $\phi$, assuming $\mu(K)>0$. Then, the $\mu$-weighted covariogram of $K$ is given by
\[
g_{K,\mu}(x)=\mu(K\cap(K+x)), \quad x\in\R^n.
\]
In general, taking $f=\varphi\cdot\chi_K$ does not give $C_f=g_{K,\mu}$. Indeed, $C_f$ is always even, while $g_{K,\mu}$ is not in general. This weighted covariogram was introduced to obtain weighted versions of Zhang's projection inequality. The weighted radial $p$th mean bodies $R_{p,\mu} K\subset \R^n$ introduced in \cite{LP25} to obtain weighted versions of the Gardner-Zhang reverse isoperimetric inequalities, can be expressed as
\[
R_{p,\mu} K = K_p(g_{K,\mu}).
\]
\noindent Since $K\cap (K+x)\subseteq K$, $g_{K,\mu}$ attains its maximum at the origin. Thus, by Proposition~\ref{p:pre}, $g_{K,\mu}\in \LC^0$. By Theorem~\ref{t:ball_expanded}, $R_{p,\mu} K$ is a convex body for all $p>-1$. 

\medskip

{\bf The Higher-Order Setting.} In \cite{Sch70}, R. Schneider introduced, for $m\in \N$, the $m$th-order covariogram $g_{K,m}:(\R^n)^m\to \R_+$ of a convex body $K\subset \R^n$,
\[
g_{K,m}(x_1,\dots,x_m)=\vol_n(K\cap_{i=1}^m(K+x_i)), \quad x_i\in\R^n,
 \]
to study a generalization of the difference body of $K$. Indeed, he defined $D^m(K)\subset (\R^n)^m$ as the support of $g_{K,m}$. Higher-order radial $p$th mean bodies $R_p^m K\subset(\R^n)^m$ were introduced in \cite{HLPRY25} to obtain a generalization of Zhang's projection inequality, and, more generally, the Gardner-Zhang inequalities. In the current terminology, they are precisely
\[
R^m_p K = K_p(g_{K,m}).
\]
In \cite{LPRY25}, this development was combined with the weighted setting: assuming $\mu(K)>0$ and defining $$g_{K,m,\mu}(x_1,\dots,x_m)=\mu(K\cap_{i=1}^m(K+x_i)), \quad x_i\in\R^n,$$ it holds
\[
R_{p,\mu}^m K = K_p(g_{K,m,\mu}).
\]
In \cite{LMU25}, the $L^1$ covariogram of a function was extended to the higher-order setting,
\[
g_{f,m}(x_1,\dots,x_m)=\int_{\R^n}\min\{f(y),f(y-x_1),\dots,f(y-x_m)\}\dlat y, \quad x_i\in\R^n,
\]
and then the $m$th-order $L^1$ radial $p$th mean bodies of $f$ were defined via $K_p(g_{f,m})\subset (\R^n)^m$.

\noindent Each of the covariograms $g_{K,m},g_{K,m,\mu},g_{f,m}$ belongs to $\operatorname{LC}_{nm}^0$. Consequently, Theorem~\ref{t:ball_expanded} yields the result that the bodies $R^m_p K,R^m_{p,\mu} K$ and $K_p(g_{f,m})$ are convex bodies for all $p>-1$.

\medskip

{\bf Acknowledgments:}
We would like to thank Artem Zvavitch, Matthieu Fradelizi, Dario Cordero-Erausquin, Tomasz Tkocz, Michael Roysdon and Juli\'an Haddad for the fruitful conversations over the past five years on radial $p$th mean bodies and log-concave functions. We especially thank Richard Gardner for the conversations in May and June 2024 and for encouraging us to pursue the extension of Ball's bodies to the range $p \in (-1,0)$.

{\bf Funding:} 
D. Langharst was funded by the U.S. National Science Foundation's MSPRF fellowship via NSF grant DMS-2502744.

\bibliographystyle{siam}

\end{document}